\documentclass[12pt,letterpaper]{article}
\usepackage{graphicx}
\usepackage{amsmath}
\usepackage{amsfonts}
\usepackage{amsthm}
\usepackage{amssymb}
\usepackage{color}

\usepackage{verbatim} 
\usepackage{hyperref}

\newtheorem{theorem}{Theorem}[section]
\newtheorem{lemma}[theorem]{Lemma}
\newtheorem{corollary}[theorem]{Corollary}

\newcommand{\be}{\begin{equation}}
\newcommand{\ee}{\end{equation}}
\newcommand{\bea}{\begin{eqnarray}}
\newcommand{\eea}{\end{eqnarray}}
\newcommand{\beas}{\begin{eqnarray*}}
\newcommand{\eeas}{\end{eqnarray*}}

\begin{document}

\title{Saturated Domino Coverings}

\author{Andrew Buchanan\\Hong Kong \and Tanya Khovanova\\MIT \and Alex Ryba\\Queens College, CUNY}
\date{}
\maketitle

\begin{abstract}
A domino covering of a board
is saturated if no domino is redundant.  
We introduce the
concept of a fragment 
tiling and show that a
minimal fragment tiling always corresponds to a
maximal saturated domino covering.
The size of a minimal fragment tiling is
the domination number of the board.
We define a class of regular boards and show that
for these boards the domination number gives the size of
a minimal X-pentomino covering. 
Natural sequences that count maximal saturated domino coverings
of square and rectangular boards are obtained.  These include
the new sequences 
A193764, A193765, A193766, A193767, and A193768 of OEIS.
\end{abstract}

\section{Introduction}

We consider domino coverings of \emph{boards}.  A board is
a finite set of cells from a square grid. In order to 
have a domino covering every cell of the board should 
be connected to another cell. From now on we only consider such boards.
Our work is motivated by the following problem,
which was given
to us by Rados Radoicic:

\begin{quote}
A 7 by 7 board is covered with 38 dominoes such that each covers 
exactly 2 squares of the board. Prove that it is possible to remove 
one domino so that the remaining 37 still cover the board. 
\end{quote}

Call a domino covering of a board \emph{saturated} if the 
removal of any domino leaves an 
uncovered cell. For a board $B$, let $d(B)$ be 
the number of dominoes in its 
largest saturated covering.  
In particular, when $B$ is an $n \times n$ board
we write $d(n)$ for $d(B)$. 
The problem above asks us to prove that $d(7) < 38$.

In Section~\ref{sec:crosses} we consider two other sorts
of covering. In an X-pentomino covering the tiles are allowed
to overlap and poke out of the board. 
We also consider fragment tilings, in which each tile is a
connected fragment formed from two or more cells of an X-pentomino.
Here, tiles may not overlap or poke outside the board.
We reserve the word \emph{tiling} to mean a 
covering where the tiles do not overlap.

In a fragment tiling,
we can assign a center to each tile.  These centers
form a dominating set for the board.
Dominating sets have been much studied 
\cite{jac, cock, cha, cha-cla, hares, cha-cla-ha, spa, gui, ala}.
A recent preprint \cite{gon} establishes an exact formula for the
size of a minimal dominating set in a rectangular grid.
We prove that minimal fragment tilings correspond to maximal saturated
domino coverings for all boards and correspond to X-pentomino coverings
for regular boards (which include the well studied grids).
In particular this shows that $d(n) = n^2 - \gamma(n)$ where $\gamma(n)$ is 
the domination number of an $n \times n$ board.

In Section \ref{sec:examples}
we determine $d(n)$ for a few small boards. 
In Section~\ref{sec:largeboards} we briefly discuss large boards
and in Section~\ref{sec:aliens} we consider analogous results
for boards on grids made of
triangular and hexagonal cells.  Finally, in Section \ref{sec:dominum}
we discuss domination numbers.

\section{Domino Coverings and X-Pentomino Coverings}\label{sec:crosses}

For any saturated domino covering of any board,  we construct a 
graph on the set $D$ of cells by joining two if a domino links them.
In this graph, any vertex whose degree exceeds 1 
can be adjacent only to vertices
with degree 1. (Otherwise a domino that joins two 
vertices with larger degree
is redundant and can be removed from the 
supposedly \emph{saturated covering}.)
This condition means that connected components 
of our graph are stars.  Since a
square on the board has at most 4 neighbors, 
the only possible connected components correspond to collections
of dominoes with the following shapes (see Figure~\ref{fig:fragments}). 
We term a vertex of maximal degree in a star a center. 
Note that the smallest fragment has two centers.

\begin{figure}[htpb]
\centering
\includegraphics[scale=0.5]{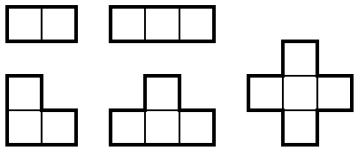}
\caption{Fragments}
\label{fig:fragments}
\end{figure}

The largest shape in the figure is the X-pentomino and the next
is the T-tetromino. All these shapes 
are \emph{fragments} of an X-pentomino. We have thus established

\begin{theorem}
A saturated domino covering of a board gives a decomposition
of the cells of the board into a set of non-overlapping fragments.
\end{theorem}

Now consider a board with $\beta$ cells that
admits a saturated covering 
by $\delta$ dominoes forming a set of $\phi$ fragments. 
Observe that any fragment with $f$ cells consists of exactly
$f - 1$ dominoes. Summing over the fragments, we deduce that 
$\delta = \beta - \phi$. 
Thus, in order to make $\delta$ as 
large as possible, we should make $\phi$ as small as possible. 
In other words, a largest saturated domino covering is obtained from
a minimal fragment tiling.  For a board $B$, write 
$|B|$ for the number of cells 
and 
$f(B)$ for the size of a minimal fragment tiling. We have

\begin{theorem}\label{thm:df}
$d(B) = |B| - f(B)$.
\end{theorem}

Although we insist on non-overlapping tiles in our
fragment tilings, the following lemmas show that smaller coverings
are not available even if we allow tiles to overlap.

\begin{lemma}\label{lem:overlap}
Two overlapping fragments can be replaced by either a single fragment
or two disjoint fragments that cover the same set of squares.
\end{lemma}

\begin{proof}
If there exists a cell that is a center for both fragments, then their union is
a fragment.
Otherwise, suppose the two fragments have adjacent centers. 
Then no other cell can belong to both fragments, and we can remove 
the link between the two centers to give us two disjoint fragments. 
The final possibility is that the fragments have neither shared nor 
adjacent centers. In this case we just assign the shared cells to 
one or other center arbitrarily, ensuring that each fragment 
contains at least two cells.
\end{proof}

Observe that if a fragment is adjacent to an isolated square, that isolated
square can be replaced by a domino that overlaps the fragment. 
Lemma \ref{lem:overlap} applies and we deduce

\begin{lemma}\label{lem:isolate}
An isolated square and an adjacent fragment 
can be replaced by either a single fragment
or two disjoint fragments that cover the same set of squares.
\end{lemma}

We now turn our attention to X-pentomino coverings of a board.
Note that for a fragment tiling, we insist that the fragments are disjoint and
are completely contained in the board under consideration.  For X-pentomino
coverings we relax both of these requirements.

\begin{theorem}\label{thm:x<=f}
The smallest number of X-pentominoes that can cover 
a board $B$ allowing overlaps and tiles that poke outside 
is not greater than
$f(B)$.
\end{theorem}

\begin{proof}
Denote the smallest number of X-pentominoes that can cover the 
board by $x(B)$. In a covering by fragments, each tile can be 
expanded to an X-pentomino, hence $f(B) \geq x(B)$. 
\end{proof}

For many boards, including rectangular boards, $f(B)=x(B)$. 
Define a board to be \emph{regular} 
unless it contains a pair of cells that have no common
neighboring square on the board but do have one in the ambient grid.

\begin{theorem}\label{thm:x=f}
If $B$ is a regular board then $x(B) = f(B)$.
\end{theorem}

\begin{proof}
We show that a minimal X-pentomino covering can
be transformed to a fragment tiling without
increasing the number of tiles.  We may first
trim any parts of tiles that poke out of the board.
It is possible that such a trimming might replace a
tile by two isolated cells.  In this case, regularity
of the board guarantees that these cells have a common neighbor 
belonging to the board. 
We join the common neighbor to the two isolated cells to form a fragment which we
use in their place.  This ensures that trimming does not
increase the number of tiles.
After the trimming we have a cover by (possibly overlapping)
fragments and isolated
cells.  Lemmas \ref{lem:overlap} and \ref{lem:isolate} are now
applied to remove the overlaps and absorb the isolated cells.
\end{proof}

For an arbitrary board $B$ the size $x(B)$ of a minimal
X-pentomino cover is not necessarily equal to the size $f(B)$ 
of a minimal fragment tiling. In
the following pair of examples (see Figure~\ref{fig:counter}), in 
which the truncation
of an X-pentimino that pokes out of the board leaves
two isolated squares, we have $x(B) < f(B)$.

\begin{figure}[htpb]
\begin{center}$
\begin{array}{cc}
\includegraphics[scale=0.3]{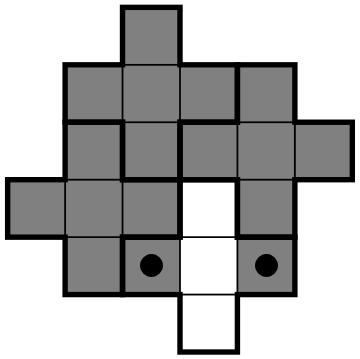} &
\includegraphics[scale=0.3]{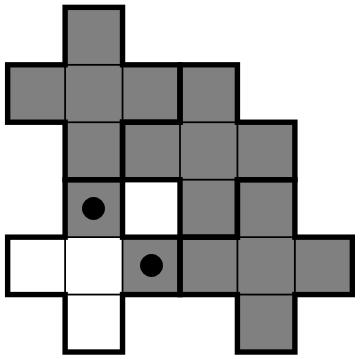}
\end{array}$
\end{center}
\caption{Examples of Irregular Boards}
\label{fig:counter}
\end{figure}

In both examples, the boards are gray. 
The two 
offending cells are marked with dark disks. 
To see that 4 fragments cannot tile either board, 
observe that the number of cells is 17 in both cases. This implies that
an X-pentomino must be used in the tiling. 
It is easy to check that no location for an X-pentomino 
allows the completion of a tiling with only 4 fragments.

Although regularity of the board is sufficient to guarantee that
$f(B) = x(B)$,
it is not necessary. For example, if a middle cell is
removed from  a $3 \times 2$ rectangle,
we are left with an irregular board for which $f(B) =x(B) = 2$.

The earlier analysis applies to any board,
we now specialize to square boards.
Let $f(n)$ be the minimal number of fragments that are 
required to cover an $n \times n$ board without overlap. 
Let $x(n)$ be the minimum number of X-pentominoes that can cover 
an $n \times n$ square board (where overlap and poking out are allowed).  
An $n \times n$ board is regular so our earlier theorems give

\begin{corollary} 
$d(n) = n^2 - f(n)$.
\end{corollary} 

\begin{corollary} 
$x(n) = f(n)$.
\end{corollary}

\section{Small Examples}\label{sec:examples}

We now consider some small examples.
First we prove that $d(2) = 2$. 
We must show that 3 dominoes on a $2 \times 2$ board include
a redundancy. Let us 
rotate the board so that at least two of the dominoes are horizontal. 
If they coincide, then we can remove one of them. If not, they 
completely cover the board and we can remove the third one. 

To show that $d(3) = 6$ it is simplest to prove that $f(3) = 3$. 
A fragment tiling with size 2 would have to use an
X-pentomino (because $9 > 4 \times 2$).  This would cut off
all 4 corners preventing the placement of further fragments.
In Figure~\ref{fig:small}, we
exhibit a covering with 3 fragments. 

For $n = 4$, we have $f(4) = 4$. 
Again, we cannot use an X-pentomino in a fragment tiling because it would 
cut out a corner. As it is possible to cover the board with 
four T-tetrominoes, it must be the best covering.

For $n = 5$, we show that $f(5) = 7$.
There are five ways to place an X-pentomino so that 
a corner is not cut off. In all of these ways the X claims the 
center point. Hence, two X-pentominoes cannot be used in a covering. 
Any tiling with six pieces would need one X and five Ts. In 
addition, if we consider the checkerboard coloring, each T-tile 
has 1 square of one color and 3 of the other color. Hence, 
five T-tiles have an odd number of each color. 
Hence the X-pentomino cannot be right in the center, 
and must instead just cover the center cell with a 
branch. Two of the corner cells are now uncoverable. 
Thus, we cannot cover the board with 6 fragments. 
A covering with 7 fragments is demonstrated in Figure~\ref{fig:small}.

\begin{figure}[htpb]
\centering
\includegraphics[scale=0.5]{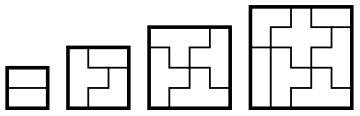}
\caption{Small Examples}
\label{fig:small}
\end{figure}

For $n=6$, it is impossible to tile
with 9 fragments.  Otherwise, the average number of 
cells in a fragment is 4. 
By parity, the fragments cannot all be T-pentominos.
Thus the number of fragments
that have fewer
than 4 cells cannot be larger than the number of X-pentominoes. 
But every X-pentomino needs two small fragments at the 
closest corner. Hence the tiling is impossible.  Our example of
a tiling in Figure~\ref{fig:covering7} with 
10 fragments establishes that $f(6) = 10$.

In our original problem of the 7 by 7 square, we
claim that $f(7) = 12$. The square 
consists of an inner 5 by 5 square surrounded by a rim. We first 
prove that we cannot use six X-pentominoes in a tiling of the 
square. The center of any X-tile must be in the inner square 
and cannot be at a corner of this square (or a corner of the larger 
square would be cut off). However six X-pentominoes cover 30 squares, 
so five of them would have to poke out into the rim. Two of these 
five would have to be centered on the same side of the inner 
square, and the three available locations on that side do not 
provide enough room for the center of two X-tiles.

So we can have at most five X-tiles. Therefore a tiling with 
11 fragments would have to use five X-tiles and six T-tiles. 
However, only T-tiles can be used to cover the corners of the 
board, and it is easy to see that the $2\times 2$ region at 
each corner requires at least two T-tiles. Therefore eight 
T-tiles must be used and there is no tiling with 11 fragments.

We have proved that it is impossible to cover the board with 
11 tiles. A covering with 12 tiles is shown in 
Figure~\ref{fig:covering7}.

\begin{figure}[htpb]
\centering
\includegraphics[scale=0.5]{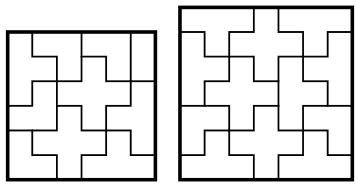}
\caption{Tilings of $6 \times 6$ and $7 \times 7$ Boards}
\label{fig:covering7}
\end{figure}

We have shown that the sequence $f(n) = x(n)$ starts as 
2, 3, 4, 7, 10, 12, $\ldots$.
Correspondingly, $d(n)$ = 2, 6, 12, 18, 26, 37, $\ldots$.
The latter sequence provides the new entry A193764
in OEIS. The sequence corresponding to the original puzzle is now the sequence
A193765: $A193765(n) = A193765(n) +1$.  It is the size of the
smallest set of dominoes for which any covering of an $n$ by $n$ board
involves a redundant domino.

Our sequence for $f(n)$ matches sequence A104519 of T. Suzuki \cite{OEIS}.
The matched sequence has been extended up to
$f(29) = 188$ \cite{ala}.
Suzuki's sequence is defined to count the number of
monominoes needed to exclude X-pentominoes from an $n \times n$
board.  The sequence also counts the minimal number of
X-polynominoes needed to cover an $n \times n$ board (i.e. $x(n)$).

\section{Large Boards}\label{sec:largeboards}

Let us look at large boards. 
Obviously $f(n) \geq n^2/5$ since the largest tile has size 5. 
One can easily obtain an upper bound for $x(n)$, which in view 
of our Theorem~\ref{thm:x=f} is also an upper bound for $f(n)$. Consider 
an infinite tessellation of the square grid using X-pentominoes. 
If an $n$ by $n$ square in the grid meets exactly $k$ pentominoes, 
then we have $x(n) \leq k$.
There are only five inequivalent ways to place an $n \times n$ square
onto the tesselation.  We obtain five upper bounds for $f(n)$ of which the smallest
gives $f(n) \leq \frac{(n+2)^2 - k(n)}{5},$  where $k(n)$ is $4,4,6,5,16$
according as $n$ mod 5 is  $0, 1, 2, 3, 4$.

In general, the restriction of the tessellation to a finite 
square of tiles gives an inefficient choice of covering pentominoes 
on the rim of the square. For example, in the following restriction 
of the tessellation to a 5 by 5 square, the use of 2 pentominoes 
denoted $H$ and $I$ to cover 2 adjacent cells on the bottom row is 
wasteful (see Figure~\ref{fig:xpentomino}). It would be better to 
place a single pentomino so as to cover these cells. In a similar 
way the pentominoes $C$ and $F$ that cover just one cell each in 
the right column can be replaced by a single pentomino that cover 
both of these cells and overlaps with pentomino $B$ on other cells. 
This improves our tesselation derived upper bound for $x(5)$ from 9 to 7.

\begin{figure}[htpb]
\centering
\includegraphics[scale=0.5]{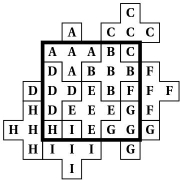}
\caption{An X-pentomino Covering of a $5 \times 5$ Board}
\label{fig:xpentomino}
\end{figure}

We have demonstrated that 
for rectangular boards that are large in both directions 
$x(B)$ 
is of the order of $|B|/5$. This is indeed true and the exact number is 
known as we explain in Section~\ref{sec:dominum}.

\section{Other Grids}\label{sec:aliens}

We can generalize our main theorems to other grids. 
Consider a grid made of triangles. Here the analogue of a domino is a tile  
made of two adjacent triangles. Fragments now have one or two
centers and 1, 2 or 3 spokes, see Figure~\ref{fig:trfr}.

\begin{figure}[htpb]
\centering
\includegraphics[scale=0.5]{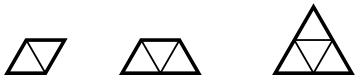}
\caption{Fragments for a Triangular Grid}
\label{fig:trfr}
\end{figure}

Similarly, for a hexagonal grid a domino (analogue) is a shape made of two adjacent 
hexagons. The fragments can have 
up to three centers and up to 6 spokes, see Figure \ref{fig:hfr}. 
Note that in the hexagonal case, there are three 
``multi-yolk'' fragments with multiple centers. 

Saturated domino coverings in the rectangular grid are related
to coverings by X-pentominoes. For triangular and hexagonal grids, the role of 
the X-pentomino is played by the maximal fragment. 
For a board $B$ on a triangular or hexagonal grid, we define $d(B)$
as the size of a maximal
saturated cover by domino analogues. 
The values of $f(B)$ and $x(B)$ are also defined by analogy with the
case of the rectangular grid. Thus $f(B)$ is the size of a minimal fragment tiling and
$x(B)$ is the minimal size of 
a covering with maximal fragments (allowing tiles to overlap
or poke out). 

\begin{figure}[htpb]
\centering
\includegraphics[scale=0.5]{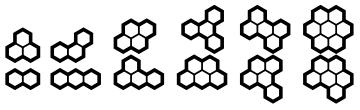}
\caption{Fragments for a Hexagonal Grid}
\label{fig:hfr}
\end{figure}

As before we assume that boards do not have any isolated cells.
For both of these new types of grid, the 
analogue of Theorem \ref{thm:df} holds:

\begin{theorem}
If $B$ is a board in a triangular or hexagonal grid, then $d(B) = |B| - f(B)$.
\end{theorem}

The proof of the following theorem is similar to the proof of 
Theorem \ref{thm:x<=f} for rectangular boards. 

\begin{theorem}\label{thm:xf}
If $B$ is a board in a triangular or hexagonal grid, then $x(B) \geq f(B)$.
\end{theorem}

A board made of triangles is \emph{regular} if 
for any two non-adjacent cells with a common neighbor in the ambient grid, 
that neighbor belongs to the board. Equivalently the board 
does not contain exterior angles of 60 degrees.

In particular, if a maximal fragment 
meets a regular board,
the intersection cannot contain two different spokes without the center. 
Hence the intersection is 
either a 
fragment or a single cell. This is enough to prove 
the analogue of Theorem~\ref{thm:x=f}:

\begin{theorem}
If $B$ is a regular board on a triangular grid, then $x(B) = f(B)$.
\end{theorem}

Figure~\ref{fig:counter3} illustrates an example of an irregular board 
that can be covered by two maximal fragments one of which pokes out. 
However, it
requires three fragments to be tiled. 
As in our examples for square grids, the board is gray and the 
offending cells are marked by disks.

\begin{figure}[htpb]
\centering
\includegraphics[scale=0.2]{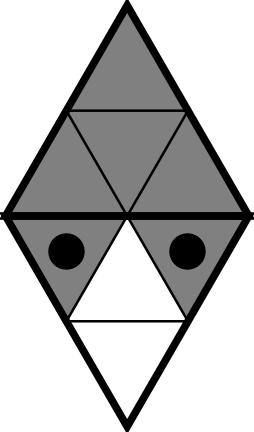}
\caption{An Irregular Triangular Board}
\label{fig:counter3}
\end{figure}

Analogously, a hexagonal board is called \emph{regular} if for any 
two hexagonal cells 
that have a common neighbor, at least one of the 
neighbors belongs to the board. 
For such boards the analogue of Theorem~\ref{thm:x=f} holds:

\begin{theorem}
If $B$ is a regular board on a hexagonal grid, then $x(B) = f(B)$.
\end{theorem}

The left example in Figure~\ref{fig:counter4} contains two marked
hexagonal cells that have exactly one common neighbor 
which does not belong to the board. The board can be covered with one 
fragment that pokes out, but it needs two fragments to be 
tiled without any poking out. The right example has two marked hexagonal 
cells that have two neighbors, neither of which belongs 
to the board. The right board can be covered 
with three fragments if poking out is allowed and four otherwise.

\begin{figure}[htpb]
\centering
\includegraphics[scale=0.5]{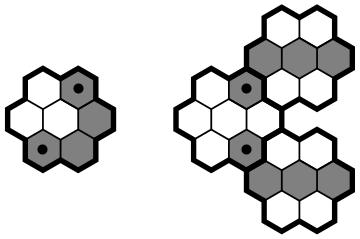}
\caption{Irregular Hexagonal Boards}
\label{fig:counter4}
\end{figure}

The natural lower bound for the size of a minimal fragment tiling of 
a triangular grid is the size of the board divided by 4.  This bound
is exact for triangular boards with even side length (since these 
can be tiled
with maximal fragments). 
For a hexagonal grid 
the natural lower bound is the size divided by 7.

\begin{theorem}
For a triangular board with side $n$ in a triangular grid 
the size of a minimal cover by maximal fragments is
$x(n) = \lceil n^2/4 \rceil$. 
\end{theorem}

The sequence $x(n)$ arises in many other contexts. It appears in 
OEIS as sequence A004652.  However, our interpretation of the sequence is new.

\begin{proof}
It is obvious that
$\lceil n^2/4 \rceil$ is a lower bound since the board has
$n^2$ cells and each tile covers 4 cells.
When $n$ is even, the board is completely covered by tiles without overlap
in the obvious way. This gives
$x(n) = n^2/4 $.  Hence, we need only exhibit appropriate
coverings for boards with odd sizes.

\begin{figure}[htpb]
\centering
\includegraphics[scale=0.35]{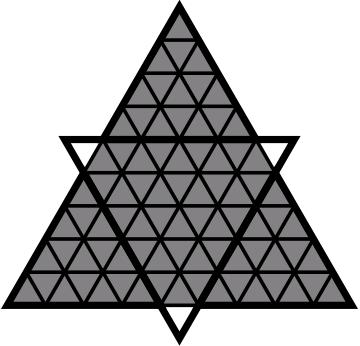}
\caption{Triangular Board with Side $4m + 1$} 
\label{fig:tri1}
\end{figure}

In the case where $n = 4m+1$,  
we partition the board into three triangles
with side $2m$ at the corners and a central piece that is almost a
triangle with side $2m+2$.  
(See Figure~\ref{fig:tri1}).
Indeed, if we 
attach 3 cells to the central piece it does become a triangle of 
size $2m+2$.  The three corner pieces and the enlarged center are
all triangles with even side.  They can therefore be covered using
a total of $3m^2 + (m + 1)^2 = 4m^2 + 2m + 1 = \lceil n^2/4\rceil$ tiles.

\begin{figure}[htpb]
\centering
\includegraphics[scale=0.35]{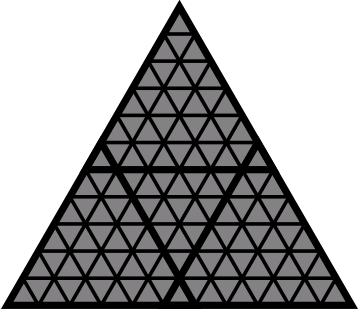}
\caption{Triangular Board with Side $4m + 3$} 
\label{fig:tri3}
\end{figure}

Similarly, when
$n$ has the form $4m+3$, we place triangular pieces with side
$2m + 2$ at the three corners of the board.  (These pieces overlap
at 3 cells.)  A triangular central piece with side $2m$ completes
the board.  
(See Figure~\ref{fig:tri3}).
The 4 pieces can be covered with
$3(m + 1)^2  + m^2 = 4m^2 + 6m + 3 = \lceil n^2/4\rceil$ tiles.
\end{proof}

\section{The Domination Number}\label{sec:dominum}

We can take this one step further and 
consider any connected graph. An edge (and its end points)
now play the role of a
\emph{domino}. The following analogue of Theorem \ref{thm:df} holds:

\begin{theorem}
The largest number of dominoes in a saturated domino covering of a connected
graph is equal to the number of vertices minus the minimum number 
of star graphs that can cover the graph.
\end{theorem}

The minimum number of star graphs that can cover a graph $G$ is 
called the \emph{domination number} and is 
denoted $\gamma(G)$. This notion is well-studied, see \cite{hed}. 

Suppose we have a board $B$. The adjacency graph of
$B$ has vertices that correspond to cells of $B$. 
Two vertices are joined by an edge if 
they are adjacent on a board. We immediately obtain

\begin{theorem}
Let $B$ be a board with adjacency graph $Ad(B)$, 
then $d(B) = |B| - \gamma(Ad(B))$.
\end{theorem}

The adjacency graph of a rectangular $m \times n$ board on a square grid 
is known as the grid graph $G_{m,n}$.

\begin{corollary}
If $B$ is a rectangular $m \times n$ board, then $d(B) = |B| - \gamma(G_{m,n})$.
\end{corollary}

It was conjectured by Chang \cite{cha} and recently proved in \cite{gon} 
that for $16 \leq n \leq m$, the domination number of an $m \times n$
grid is: 
\begin{equation}
\gamma(G_{n,m}) = \lfloor \frac{(n+2)(m+2)}{5} \rfloor -4.
\end{equation}

For other rectangular grids the dominating number is also known \cite{ala}. 
Combining known results \cite{gon}, \cite{ala} gives the full completion 
A104519 of 
the sequence $x(n)$: 

\[
1, 2, 3, 4, 7, 10, 12, 16, 20, 24, 29, 35, 40, 47, 53, 60, 68, 76, 84, \ldots \]

The term $x(n)$ is given by $4\lfloor (n+2)^2/5 \rfloor$ for any $n \geq 6$ and $n \neq 13$. 

We briefly discuss the sequences that count
minimal fragment tilings of $m \times n$ rectangular 
boards for small values of $m$, or equivalently $n$. For 
simplicity, we will denote the board as $m \times n$.

Suppose $m = 1$. Then the best thing to do is to cover the board by 1 by 3 
rectangles and $f(1 \times n) = x(1 \times n) = \gamma (G_{1,n}) = \lceil n/3 \rceil$ (A008620). 
Correspondingly, $d(1 \times n) = n - \lceil n/3 \rceil$ (A004523).

When one of the dimensions is between 2 and 4 inclusive, we
show that $f(m \times n) = x (m \times n) = \gamma (G_{m,n}) \geq \frac{mn}{4}.$

Clearly a board of size  2 by $n$ cannot contain an X-pentomino, 
so the best we can do is to use T-tiles, and at least 
$\frac{mn}{4}$ are required. The explicit formula is known: 
$f(2 \times n) = f(2 \times n) = \gamma (G_{2,n})= \lfloor (n+2)/2 \rfloor$ (A008619). 
Correspondingly, $d(2 \times n) = 2n - \lfloor (n+2)/2 \rfloor$ (A001651).

Now, if a $3 \times  n$ board has a fragment tiling that includes an
X-pentomino, then that X-pentomino requires two different fragments to
the right of it. Moreover, they cannot both be T-tiles. Hence, 
we cannot do better then 4 squares per tile on average. 
The explicit formula is: $f(3 \times n) = f(3 \times n) = \gamma (G_{3,n})= \lfloor (3n+4)/4 \rfloor$ (A037915). 
Correspondingly, $d(3 \times n) = 3n - \lfloor (3n+4)/4 \rfloor$ 
(which was added as A193766).

A similar argument applies to the $4 \times n$ case.
Suppose there is an X-pentomino involved in a fragment tiling. 
Up to reflection, there is just
one way to place the X-pentomino. The cells adjacent 
to it cannot only belong to X-pentominoes 
and T-tiles. Hence, every X-pentomino 
requires at least one companion small fragment. 
If two X-pentominoes are joined they require 
two small fragments. Again we cannot exceed an average of
4 squares per tile. The explicit formula is: 
$f(4 \times n) = f(4 \times n) = \gamma (G_{4,n}) = n+1$, if $n = 1,2,3,5,6,9$, 
and $n$ otherwise (which was added as A193768). 
Correspondingly, $d(4 \times n) = 4n - 1 \text{ or } 4n$ 
(which was added as A193767).

We can easily see that for narrow $m \times n$ rectangles where 
the second dimension is large, we
can approach arbitrarily close to the bound $mn/4$ by truncating the
tessellations made by repeating the following patterns 
(See Figure~\ref{fig:longs}).

\begin{figure}[htpb]
\centering
\includegraphics[scale=0.5]{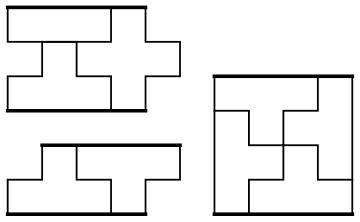}
\caption{Tessellations}
\label{fig:longs}
\end{figure}

\bigskip
\hrule
\bigskip

\noindent 2000 {\it Mathematics Subject Classification}: Primary
05C69, 
05A15, 
05B45, 
11B99, 
52C20. 

\noindent \emph{Keywords: domino coverings, dominating sets, sequences.} 

\bigskip
\hrule
\bigskip

\noindent
(Mentions A001651, A004532, A004652, A008619, A008620, A037915, 
A104519, A193764, A193765, A193766, A193767, A193768.)

\bigskip
\hrule
\bigskip


\begin{thebibliography}{9}

\bibitem{ala}
\mbox{Samu Alanko, Simon Crevals, Anton Isopoussu, Patric \"Ostergard, Ville Pettersson},
Computing the Domination Number of Grid Graphs,
{\em The electronic journal of Combinatorics}
{\bf 18} (2011).

\bibitem{cha}
\mbox{T.Y. Chang},
Domination Numbers of Grid Graphs,
{\em Ph.D. Thesis, Department of Mathematics, University of South Florida,\ }
(1992).

\bibitem{cha-cla}
\mbox{T.Y. Chang, W.E. Clark},
The domination numbers of the $5 \times n$ and $6 \times n$ grid graphs,
{\em J. Graph Theory}
{\bf 17} (1993), 81--108.

\bibitem{cha-cla-ha}
\mbox{T.Y. Chang, W.E. Clark, E.O. Hare},
Dominations of complete grid graphs I,
{\em Ars Combin.}
{\bf 38} (1994), 97--112.

\bibitem{cock}
\mbox{E.J. Cockayne, E.O. Hare, S.T. Hedetniemi, T.V. Wimer},
Bounds for the Domination Number of Grid Graphs,
{\em Congressus Numeratium}
{\bf 47} (1985), 217--228.

\bibitem{gon}
\mbox{Daniel Goncalves, Alexandre Pinkou, Micha\"el Rao, St\'ephan Thomass\'e},
The domination number of grids,
{\em Preprint}
(2011).


\bibitem{gui}
\mbox{David R. Guichard},
A lower bound for the domination number of complete grid graphs,
{\em J. Combin. Math. Combin. Comput.}
{\bf 49} (2004), 215--220.

\bibitem{hed}
\mbox{S. T. Hedetniemi, R. C. Laskar},
Bibliography on domination in graphs and some basic definitions of 
domination parameters, 
{\em Discrete Mathematics}
{\bf 86} (1--3) (1990), 257--277.

\bibitem{jac}
\mbox{M. S. Jacobson, L. F. Kinch},
On the domination number of products of a graph; I,
{\em Ars Combin.}
{\bf 10} (1983), 33--44.

\bibitem{hares}
\mbox{E. O. Hare and W. R. Hare},
Domination in graphs similar to grid graphs, 
{\em Congr.  Numer.}
{\bf 97} (1993), 143--154.

\bibitem{OEIS} N. J. A. Sloane, Online Encyclopedia of Integer Sequences (OEIS). \url{http://www.research.att.com/~njas/sequences/}

\bibitem{spa}
\mbox{A. Spalding},
Min-Plus Algebra and Graph Domination, 
{\em Ph.D. thesis, Dept. of Applied Mathematics University of Colorado,\ }
(1998).


\end{thebibliography}
\end{document}